\documentclass[12pt,twoside]{amsart}

\usepackage{amssymb}
\usepackage{graphicx}
\usepackage[pdfauthor={David Cook II, Eastern Illinois University},
            pdftitle={On decomposing Betti tables and O-sequences},
            pdfsubject={Combinatorial commutative algebra},
            pdfkeywords={Pure resolutions, Betti numbers, O-sequences, Boij-Soederberg theory}
            pdfproducer={LaTeX with hyperref},
            pdfcreator={latex->dvips->ps2pdf},
            pdfpagemode=UseOutlines,
            bookmarksopen=true,
            bookmarksopenlevel=2,
            letterpaper,
            bookmarksnumbered=true,
            pdfborder={0 0 0},
            colorlinks=false]{hyperref}
\usepackage[margin=1in]{geometry}
\usepackage[foot]{amsaddr}

\def\urltilda{\kern -.15em\lower .7ex\hbox{\~{}}\kern .04em}  % URL-tilde
\newcommand{\ZZ}{\ensuremath{\mathbb{Z}}}                     % the integers
\newcommand{\mm}{\ensuremath{\mathfrak{m}}}                   % Fraktur m
\newcommand{\st}{\ensuremath{\colon\,}}                       % such that
\newcommand{\dcup}{\ensuremath{\,\mathaccent\cdot\cup\,}}     % disjoint union
\newcommand{\ddd}{\ensuremath{\dcup \cdots \dcup}}            % many disjoint unions

\numberwithin{figure}{section}
\numberwithin{equation}{section}

\newtheorem{theorem}{Theorem}[section]
\newtheorem{lemma}[theorem]{Lemma}

\newtheorem{corollary}[theorem]{Corollary}

\theoremstyle{definition}

\newtheorem{remark}[theorem]{Remark}
\newtheorem{example}[theorem]{Example}

\newtheorem*{acknowledgement}{Acknowledgement}

\begin{document}

% -- Title Block and Abstract
\title{On decomposing Betti tables and \texorpdfstring{$O$}{O}-sequences}
\author[D.\ Cook II]{David Cook II}
\address{Department of Mathematics \& Computer Science, Eastern Illinois University, Charleston, IL 61920, USA}
\email{\href{mailto:dwcook@eiu.edu}{dwcook@eiu.edu}}
\subjclass[2010]{13F55, 05E45, 13D02, 05C15, 06A12}
\keywords{Pure resolutions, Betti numbers, $O$-sequences, Boij-S\"oderberg theory}

\begin{abstract}
    The Boij-S\"oderberg characterization decomposes a Betti table into a unique positive integral linear combination of pure diagrams.
    Given a module with a pure resolution, we describe explicit formul\ae\ for computing the decomposition of the Betti table of the module
    given the decomposition of the truncation of the Betti table, and vice versa.
    
    Nagel and Sturgeon described the decomposition of Betti tables of ideals with $d$-linear resolutions; indeed, the coefficients are
    precisely finite $O$-sequences.  Using the extension formul\ae, we provide an explicit description of the coefficients of the 
    decomposition of the Betti table of the quotient ring of such an ideal.  Following from this, we describe the closed convex 
    simplicial cone of $O$-sequences.
\end{abstract}

\maketitle

% -----------------------------------------------------------------------------
% -- Section
\section{Introduction}\label{sec:intro}

Boij and S\"oderberg~\cite{BS-2008} conjectured a complete characterization, up to multiplication
by a positive rational number, of the structure of Betti tables of finitely generated Cohen-Macaulay $R$-graded modules.
The conjecture was proved in characteristic zero by Eisenbud, Fl{\o}ystad, and Weyman~\cite{EFW} and in positive characteristic
by Eisenbud and Schreyer~\cite{ES}.  Furthermore, it was recently extended to the non-Cohen-Macaulay
case by Boij and S\"oderberg~\cite{BS-2012}.

The Boij-S\"oderberg characterization decomposes a Betti table into a unique positive integral linear combination of pure diagrams;
see Section~\ref{sub:betti}.  As the theory is still under development, there is a dearth of families with known explicit Boij-S\"oderberg decompositions.
Gibbons, Jeffries, Mayes, Raicu, Stone, and White~\cite{GJMRSW} described the decomposition for complete intersections of codimension at most three,
and Wyrick-Flax~\cite{WF} extended this to codimension four.  Recently, G\"unt\"urk\"un~\cite{Gun} described the decomposition of lex-segment ideals.
However, the results of Nagel and Sturgeon~\cite{NS} and of Engstr\"om and Stamps~\cite{EnSt} are of more interest here as both
offer combinatorial interpretations of their decompositions.  In particular, Nagel and Sturgeon showed that the coefficients of the decomposition of 
a Betti table of an ideal with a linear resolution is precisely a finite $O$-sequence.  Separately, Engstr\"om and Stamps computed the 
decomposition of all $2$-linear resolutions via the Betti tables of the Stanley-Reisner rings of threshold graphs.

The goal of this manuscript is to expand upon the known interpretations of Boij-S\"oederberg decompositions.
In Section~\ref{sec:bsd}, we recall the relevant definitions and theorems of Boij-S\"oderberg theory.
We further prove an extension and truncation lemma (see Lemmas~\ref{lem:extension} and~\ref{lem:truncation}) for a 
pure resolution.  In Section~\ref{sec:dec-oseq}, we use the extension lemma to extend the decomposition of a Betti table of an 
ideal with a linear resolution given by Nagel and Sturgeon (recalled here in Theorem~\ref{thm:ns-36}) to the decomposition
of the Betti table of the quotient; see Theorem~\ref{thm:RI-decomp}, which is simpler than the decomposition offered by 
Nagel and Sturgeon (recalled here in Theorem~\ref{thm:ns-39}).  An immediate consequence of this theorem is a family
of inequalities that all finite $O$-sequences must satisfy; this is collected in Lemma~\ref{lem:halfspaces}.  Together, these inequalities
describe a closed convex simplicial cone that contains the $O$-sequences with a bounded number of variables; see Theorem~\ref{thm:cone}.  
We also describe the maximal rays of the cone therein.

While this note was being written, the paper of Boij and Smith~\cite{BSm} appeared.  Their paper independently described the cone of $O$-sequences
using the same inequalities described herein, though their focus and approach are different.  We make more specific references in the main body of the text.
We further note that Bertone, Nguyen, and Vorwerk~\cite{BNV} provided analogous results in the specific case of Stanley-Reisner rings.

% -----------------------------------------------------------------------------
% -- Section
\section{Decomposing Betti tables}\label{sec:bsd}

% -- Subsection
\subsection{Betti tables}\label{sub:betti}~

Let $R = K[x_1,\ldots,x_n]$ be the $n$-variate polynomial ring over a field $K$, and let $M$ be a finitely
generated graded $R$-module, e.g., a homogeneous ideal $I$ of $R$ or its quotient $R/I$.  The $R$-module $M(d)$
given by the relations $[M(d)]_i := [M]_{i+d}$ is the \emph{$d^{\rm th}$ twist of $M$}.  A \emph{minimal free resolution of $M$}
is an exact sequence of free $R$-modules of the form
\[
    0 \longrightarrow F_n \longrightarrow \cdots \longrightarrow F_0 \longrightarrow M \longrightarrow 0,
\]
where $F_i$, $0 \leq i \leq n$, is the free $R$-module
\[
    \bigoplus_{j \in \ZZ} R(-j)^{\beta_{i,j}(M)}.
\]
The numbers $\beta_{i,j}(M)$ are the \emph{graded Betti numbers of $M$}, and $\beta_i(M) = \sum_{j \in \ZZ}\beta_{i,j}(M)$
is the \emph{$i^{\rm th}$ total Betti number of $M$}.  We encode the graded Betti numbers of $M$ as the \emph{Betti table}
$\beta(M) = (\beta_{i,j}(M))$.

\begin{remark}\label{rem:betti-tables}
    We write the Betti table of $M$ as a matrix with entry $(i,j)$ given by $\beta_{i,i+j}(M)$.  
    For example, if $\beta_{0,0}(M) = 1$, $\beta_{1,3}(M) = 6$, $\beta_{2,4}(M) = 7$, and 
    $\beta_{3,5}(M) = 2$, then we would write the Betti table of $M$ as
    \[
        \beta(M) = \begin{array}{c|cccc}
            \beta_{i,j} & 0 & 1 & 2 & 3 \\
            \hline
                      0 & 1 & . & . & . \\
                      1 & . & . & . & . \\
                      2 & . & 6 & 7 & 2
        \end{array}
    \]
    Notice that we represent zeros by periods.  
\end{remark}

We recall the Boij-S\"oderberg characterization of Betti tables, where we follow the notation given in~\cite{BS-2012}, which differs from the
notation given with the original conjecture in~\cite{BS-2008}.  For an increasing sequence of integers
${\bf d} = (d_0, \ldots, d_s)$, where $0 \leq s \leq n$, the \emph{pure diagram given by ${\bf d}$} is the matrix
$\pi({\bf d})$ with entries given by
\[
    \pi({\bf d})_{i,j} =
    \begin{cases}
        \displaystyle (-1)^i \prod_{k=0, k\neq i}^{s} \frac{1}{d_k - d_i}, & \text{if~} j = d_i; \\
        0, & \text{otherwise.}
    \end{cases}
\]
If $d_0 = 0$, then the \emph{normalized pure diagram given by ${\bf d}$} is the diagram
\[
    \overline{\pi}({\bf d}) = d_1 \cdots d_s \cdot \pi({\bf d}).
\]
Notice that $\overline{\pi}({\bf d})_{0,0} = 1$.

We define a partial order on the pure diagrams by $\pi(d_0, \ldots, d_s) \leq \pi(d'_0, \ldots, d'_t)$,
if $s \geq t$ and $d_i \leq d'_i$ for $0 \leq i \leq t$.
The Boij-S\"oderberg characterization states that the Betti table of a finitely generated $R$-module can be uniquely
decomposed into a positive integral linear combination of pure diagrams that form a chain in the partial order.  

\begin{theorem}{\cite[Theorem~4.1]{BS-2012}}\label{thm:BS-decomp}
    Let $R = K[x_1, \ldots, x_n]$ be standard graded, and let $M$ be a finitely generated graded $R$-module.
    There exists a unique chain of pure diagrams $\pi({\bf d}_0) < \ldots < \pi({\bf d}_t)$ and positive
    integers $a_0, \ldots, a_t$ such that
    \[
        \beta(M) = \sum_{j=0}^{t} a_j \cdot \pi({\bf d}_j).
    \]
\end{theorem}

Fixing bounds on the regularity and projective dimension of the Betti tables under consideration, the above
can be used to show that the space of such Betti tables spaces forms a convex simplicial fan spanned by the
relevant pure diagrams.  If instead we look at the space spanned by the normalized
pure diagrams, then the result is a simplicial complex which is a hyperplane section of the simplicial fan. 
See~\cite[Section~3]{BS-2012} for a full discussion of these objects.

% -- Subsection
\subsection{Extension and truncation lemmas}\label{sub:lem}~

A finitely generated graded $R$-module $M$ has a \emph{pure resolution of type ${\bf d} = (d_0, \ldots, d_s)$}, where
$d_0 < \cdots < d_t$ are integers, if $\beta_{i,j}(M) \neq 0$ precisely when $j = d_i$ for $0 \leq i \leq n$.
Let $M$ have a pure resolution of type ${\bf d}$, and let $N$ be a module such that $\beta_{i,j}(M) = \beta_{i-1,j}(N)$
for $i \geq 1$, that is, $\beta(N)$ is a \emph{truncation of $\beta(M)$}.  Similarly, we can say $\beta(M)$ is an
\emph{extension of $\beta(N)$}.

\begin{example}\label{exa:truncation}
    Let $R = K[a,b,c,d,e]$, and let $I = (ab, bc, cd, de, ae)$, i.e., $I$ is the edge ideal of the five-cycle.
    In this case, $\beta(R/I)$ and $\beta(I)$ are the following matrices, respectively.\\
        \begin{minipage}[b]{0.48\linewidth}
            \[
                \beta(R/I) = 
                \begin{array}{c|cccc}
                    \beta_{i,j} & 0 & 1 & 2 & 3 \\
                    \hline
                            0 & 1 & . & . & . \\
                            1 & . & 5 & 5 & . \\
                            2 & . & . & . & 1
                \end{array}
            \]
        \end{minipage}
        \begin{minipage}[b]{0.48\linewidth}
            \[
                \beta(I) = 
                \begin{array}{c|ccc}
                    \beta_{i,j} & 0 & 1 & 2 \\
                    \hline
                              1 & 5 & 5 & . \\
                              2 & . & . & 1
                \end{array}
            \]
        \end{minipage}\\
    We thus see that $\beta(R/I)$ is an extension of $\beta(I)$, and $\beta(I)$ is a truncation of $\beta(R/I)$.
    
    In general, the Betti table of an ideal is the truncation of the Betti table of its quotient.
\end{example}

The following extension lemma shows how to construct the Boij-S\"oderberg decomposition of the original Betti table given
the Boij-S\"oderberg decomposition of the truncated Betti table.

\begin{lemma}\label{lem:extension}
    Let $R = K[x_1, \ldots, x_n]$ be standard graded, and let $M$ have a pure resolution of type ${\bf d} = (d_0, \ldots, d_t)$.
    If the truncation $\beta'$ of $\beta = \beta(M)$ is the Betti table of a finitely generated graded $R$-module with
    Boij-S\"oderberg decomposition given by
    \[
        \beta' = \sum_{j = 1}^{t} \alpha_j \cdot \pi_{(d_1, \ldots, d_j)},
    \]
    then, assuming $\alpha_{t+1} = 0$,
    \[
        \beta = (\beta_{0, d_0} - \alpha_1) \cdot \pi_{(d_0)} + \sum_{j = 1}^{t} \Big( (d_j - d_0) \alpha_j - \alpha_{j+1} \Big) \cdot \pi_{(d_0, \ldots, d_j)}.
    \]
\end{lemma}
{\em Nota bene:} By assumption, see Theorem~\ref{thm:BS-decomp}, $\alpha_1, \ldots, \alpha_t$ are positive integers.
\begin{proof}
    Set $\gamma$ to be the diagram
    \[
        (\beta_{0, d_0} - \alpha_1) \cdot \pi_{(d_0)} + \sum_{j = 1}^{t} \Big( (d_j - d_0) \alpha_j - \alpha_{j+1} \Big) \cdot \pi_{(d_0, \ldots, d_j)}.
    \]
    Thus for $1 \leq i \leq t$, we have 
    \begin{equation*}
        \begin{split}
            \gamma_{i,d_i} &= \sum_{j=i}^{t} \Big( (d_j - d_0) \alpha_j - \alpha_{j+1} \Big) \cdot [\pi_{(d_0, \ldots, d_j)}]_{i,d_i} \\
                             &= \sum_{j=i}^{t} \frac{ (-1)^i \Big( (d_j - d_0) \alpha_j - \alpha_{j+1} \Big) }
                                                    { \prod_{k = 0, k \neq i}^{j} (d_k - d_i) } \\
                             &= \frac{1}{\prod_{k=0}^{i-1} (d_i - d_k)}
                                \sum_{j=i}^{t} \frac{ (d_j - d_0) \alpha_j - \alpha_{j+1} }
                                                    { \prod_{k = i+1}^{j} (d_k - d_i) } \\
                             &= \frac{1}{\prod_{k=1}^{i-1} (d_i - d_k)}
                                \sum_{j=i}^{t} 
                                    \left( 
                                        \frac{1}{\prod_{k = i+1}^{j} (d_k - d_i)}
                                        \frac{(d_j - d_0) \alpha_j - \alpha_{j+1}}{d_i - d_0} 
                                    \right) \\
                             &= \frac{1}{\prod_{k=1}^{i-1} (d_i - d_k)}
                                \left(
                                    \sum_{j = i}^{t}
                                        \frac{\alpha_j}
                                             {\prod_{k=i+1}^{j} (d_k - d_i)}
                                \right) \\
                             &= \sum_{j=i}^{t} \alpha_j \cdot [\pi_{(d_1, \ldots, d_j)}]_{i-1,d_i},
        \end{split}
    \end{equation*}
    which is $\beta'_{i-1,d_i} = \beta_{i,d_i}$.  It only remains to be shown that $\gamma_{0,d_0} = \beta_{0,d_0}$.  This follows as
    \begin{equation*}
        \begin{split}
            \gamma_{0,d_0} &= (\beta_{0,d_0} - \alpha_1) \cdot [\pi_{(d_0}]_{0,d_0}] + \sum_{j=1}^{t} \Big( (d_j - d_0) \alpha_j - \alpha_{j+1} \Big) \cdot [\pi_{(d_0, \ldots, d_j)}]_{0,d_0} \\
                           &= (\beta_{0,d_0} - \alpha_1) + \sum_{j=1}^{t} \frac{(d_j - d_0) \alpha_j - \alpha_{j+1}}{\prod_{k = 1}^{j} (d_k - d_0)} \\
                           &= (\beta_{0,d_0} - \alpha_1) + \alpha_1.
        \end{split}
    \end{equation*}
    Thus $\gamma = \beta$, as desired.
\end{proof}

An immediate consequence of the preceding lemma is a set of bounds on the possible coefficients of a pure diagram.

\begin{corollary}\label{cor:nzc}
    Let $\alpha_1, \ldots, \alpha_t$ and ${\bf d} = (d_0, \ldots, d_t)$ be as in Lemma~\ref{lem:extension}.
    For $1 \leq j \leq t$, $(d_j - d_0) \alpha_j - \alpha_{j+1} \geq 0$.
\end{corollary}
\begin{proof}
    The decomposition of $\beta$ in Lemma~\ref{lem:extension} is, after removal of zero terms, the Boij-S\"oderberg
    decomposition of $\beta$.  Hence by Theorem~\ref{thm:BS-decomp}, the remaining coefficients are positive integers.
\end{proof}

\begin{remark}\label{rem:one}
    If $I$ is an ideal with a pure resolution as described in Lemma~\ref{lem:extension}, then the first coefficient,
    $\alpha_1$, must be one.  This follows as $\beta_{0,0}(R/I) = 1 \geq \alpha_1 \geq 1$.
\end{remark}

Rearranging the extension lemma we get the following truncation lemma, which shows how to construct the Boij-S\"oderberg decomposition of the truncated
Betti table given the Boij-S\"oderberg decomposition of the original Betti table.  We note that Sturgeon~\cite[Theorem~3.2.2]{St}
explicitly described the decomposition of the truncation of a pure diagram.

\begin{lemma}\label{lem:truncation}
    Let $R = K[x_1, \ldots, x_n]$ be standard graded, and let $M$ have a pure resolution of type ${\bf d} = (d_0, \ldots, d_t)$.
    Suppose $\beta = \beta(M)$ has Boij-S\"oderberg decomposition given by 
    \[
        \beta = \sum_{j = 0}^{t} \delta_j \cdot \pi_{(d_0, \ldots, d_j)}.
    \]
    If the truncation $\beta'$ of $\beta$ is the Betti table of a finitely generated graded $R$-module, then
    \[
        \beta' = \sum_{j = 1}^{t} 
        \left(
            \sum_{k = j}^{t}
                \frac{\delta_k}
                     {\prod_{p = j}^{k} (d_p - d_0)}
        \right)
        \cdot \pi_{(d_1, \ldots, d_j)}.
    \]
\end{lemma}
\begin{proof}
    By Lemma~\ref{lem:extension}, we have the $t$ equations $\delta_j = (d_j - d_0) \alpha_j - \alpha_{j+1}$, for $1 \leq j \leq t$,
    in $t$ unknowns, i.e., $\alpha_1, \ldots, \alpha_t$.  Solving for the unknowns yields the stated result.
\end{proof}

% -----------------------------------------------------------------------------
% -- Section
\section{Decomposing \texorpdfstring{$O$}{O}-sequences}\label{sec:dec-oseq}

% -- Subsection
\subsection{\texorpdfstring{$O$}{O}-sequences}\label{sub:oseq}~

Given positive integers $a$ and $d$, the \emph{$d^{\rm th}$ Macaulay representation of $a$} is the unique expansion 
\[
    a = \binom{a_d}{d} + \binom{a_{d-1}}{d-1} + \cdots + \binom{a_t}{t}
\]
such that $a_d > a_{d-1} > \cdots > a_t \geq t \geq 1$. Define $a^{\langle d \rangle}$ to be the integer
\[
    \binom{a_d+1}{d+1} + \binom{a_{d-1}+1}{d} + \cdots + \binom{a_t+1}{t+1},
\]
and set $0^{\langle d \rangle} = 0$.  A sequence of nonnegative integers $(h_j)_{j\geq 0}$ is an \emph{$O$-sequence}
if $h_0 = 1$ and $h_{j+1} \leq h_j^{\langle j \rangle}$ for $j \geq 1$.  Macaulay's Theorem (see, e.g., \cite[Theorem~4.2.10]{BH}) 
characterizes $O$-sequences.

\begin{theorem}{\rm\bf (Macaulay's Theorem)}
    Given a sequence of nonnegative integers $(h_j)_{j\geq 0}$,  $(h_j)_{j\geq 0}$ is an $O$-sequence if and only if
    $(h_j)_{j\geq 0}$ is the Hilbert function of a standard graded $K$-algebra, $R/I$, i.e., $\dim_{K}[R/I]_j = h_j$ for all $j$.
\end{theorem}

Given a homogeneous ideal $I$ of a polynomial ring $R$, we define $h_{R/I}$ to be the $O$-sequence coming from the Hilbert function
of $R/I$.  If $R/I$ is artinian, then $h = h_{R/I}$ is eventually zero, i.e., there exists an integer $t$, called the \emph{socle degree},
such that $h_t > 0$ and $h_{t+1} = 0$. 

\begin{example}\label{exa:mm}
    Let $R = K[x_1, \ldots, x_n]$.  The maximal irrelevant ideal of $R$ is the ideal $\mm = (x_1, \ldots, x_n)$.  We notice that the
    $O$-sequence of $R/\mm^{t+1}$ is the finite $O$-sequence 
    \[
        \left( 1, \binom{n+1-1}{1}, \binom{n+2-1}{2}, \ldots, \binom{n+t-1}{t} \right)
    \]
    with socle degree $t$.
\end{example}

% -- Subsection
\subsection{Linear resolutions}\label{sub:linear}~

A finitely generated graded $R$-module $M$ has a \emph{$d$-linear resolution} if $\beta_{i,j}(M) = 0$ for all 
$0 \leq i \leq n$ and $j \neq i + d$.  Clearly, having a $d$-linear resolution implies having a pure resolution.
Nagel and Sturgeon classified the Boij-S\"oderberg decomposition of Betti tables of ideals with $d$-linear resolutions.

\begin{theorem}{\cite[Theorem~3.6]{NS}}\label{thm:ns-36}
    Let $R = K[x_1, \ldots, x_n]$, and consider the diagram
    \[
        \beta = \sum_{j=0}^{t} \alpha_j j! \cdot \pi(d,d+1,\ldots,d+j),
    \]
    where $\alpha_0, \ldots, \alpha_t$ are nonnegative integers and $t \leq n$.
    The following conditions are equivalent:
    \begin{enumerate}   
        \item $\beta$ is the Betti table of an ideal of $R$ with a $d$-linear resolution.
        \item $\beta$ is the Betti table of a strongly stable ideal $I$ whose minimal generators have degree $d$.
        \item $\beta$ is the Betti table of the ideal associated to a $d$-uniform Ferrers hypergraph $F$ with
            \[
                \alpha_j(F) =
                \begin{cases}
                    \alpha_j & \text{if~} 0 \leq j \leq t,\\
                    0 & \text{if~} t < j,
                \end{cases}
            \]
            where $\alpha_j = \#\{ (i_1, \ldots, i_d) \in F \st i_1 + \cdots + i_d = j + d \}$.
        \item $(\alpha_0, \ldots, \alpha_t)$ is an $O$-sequence with $\alpha_1 \leq d$.
    \end{enumerate}
\end{theorem}

\begin{remark}\label{rem:ns-36}
    Murai~\cite[Proposition~3.8]{Mu} gave the first classification of the possible Betti numbers of ideals with linear resolutions.
    Herzog, Sharifan, and Varbaro~\cite[Theorem~3.2]{HSV} provided a similar classification via a different approach.
    We note, however, that neither gave the Boij-S\"oderberg decomposition.
\end{remark}

Nagel and Sturgeon also classified the Boij-S\"oderberg decompositions of the quotients of Ferrers ideals.
We recall it here as a convenience to the reader.

\begin{theorem}{\cite[Theorem~3.9]{NS}}\label{thm:ns-39}
    Let $d \geq 2$, and let $F$ be a $d$-uniform Ferrers hypergraph on the vertex set $V_1 \ddd V_d$, where each $V_i = [n_i]$ for some
    natural number $n_i$.  The Betti table of the quotient ring $R/I(F)$ is
    \[
        \beta(R/I(F)) = \sum_{j=1}^{d} \sum_{S \in F_j} n_S \cdot k_S! \cdot \pi(0, d, \ldots, d+k_S),
    \]
    where $F_j$ is the Ferrers hypergraph
    \[
        F_j := \{(i_1, \ldots, \widehat{i_j}, \ldots, i_d) \st \text{there is some } i_j \in V_j \text{ such that } (i_1, \ldots, i_j, \ldots, i_d) \in F\},
    \]
    and for each $S = (i_1, \ldots, \widehat{i_j}, \ldots, i_d) \in F_j$
    \[
        n_S := \max\{i_j \in V_j \st (i_1, \ldots, i_j, \ldots, i_d) \in F\} \text{~and~} k_S := n_S - d + \sum_{p = 1, p \neq j}^{d} i_p.
    \]
\end{theorem}

In contrast to the above complicated decomposition that requires manipulation of a $d$-uniform Ferrers hypergraph, 
we provide a simpler decomposition.

\begin{theorem}\label{thm:RI-decomp}
    Let $I$ be an ideal of $R = K[x_1, \ldots, x_n]$ with a $d$-linear resolution.  If
    \[
        \beta(I) = \sum_{j=0}^{t} \alpha_j j! \cdot \pi(d,d+1,\ldots,d+j),
    \]
    where $\alpha_0, \ldots, \alpha_t$ are rational numbers and $t \leq n$, as described in Theorem~\ref{thm:ns-36}, then
    \[
        \beta(R/I) = \sum_{j=0}^{t} 
            \Big(  
                (d + j) \alpha_j - (j + 1) \alpha_{j+1}
            \Big)
        j! \cdot \pi(0,d,d+1,\ldots,d+j).
    \]
\end{theorem}
\begin{proof}
    This follows immediately from Lemma~\ref{lem:extension}, using $\alpha_j j!$ in place of $\alpha_j$.
\end{proof}

\begin{example}\label{exa:bsd-decomp}
    Let $I = (x^2y, x^2z, xy^2, xyz, xz^2, y^3) \subset R = K[x,y,z]$.  The Betti table of $I$ is
    \[
        \beta(I) = 
        \begin{array}{c|ccc}
            \beta_{i,j} & 0 & 1 & 2 \\
            \hline
            3 & 6 & 7 & 2 \\
        \end{array}
    \]
    By Theorem~\ref{thm:ns-36}, this decomposes to $2 \cdot 2! \cdot \pi_{3,4,5} + 3 \cdot 1! \cdot \pi_{3,4} + 1 \cdot 0! \cdot \pi_{3}$
    (see also \cite[Example~3.3]{NS}).  Hence the associated $O$-sequence is $\alpha = (1, 3, 2)$ and $d = 3$. 

    The Betti table of $R/I$ is
    \[
        \beta(R/I) = \begin{array}{c|cccc}
            \beta_{i,j} & 0 & 1 & 2 & 3 \\
            \hline
                      0 & 1 & . & . & . \\
                      1 & . & . & . & . \\
                      2 & . & 6 & 7 & 2
        \end{array}
    \]
    By Theorem~\ref{thm:RI-decomp}, the coefficients of the decomposition of $\beta(R/I)$ are
    \begin{equation*}
        \begin{split}
            j=0\colon & (3+0)\cdot 1 - (0+1)\cdot 3 = 0 \\
            j=1\colon & (3+1)\cdot 3 - (1+1)\cdot 2 = 8 \\
            j=2\colon & (3+2)\cdot 2 - (2+1)\cdot 0 = 10 \\
        \end{split}
    \end{equation*}
    Thus $\beta(R/I) = 10 \cdot 2! \cdot \pi(0,3,4,5) + 8 \cdot 1! \cdot \pi(0,3,4)$.
\end{example}

The $O$-sequences of quotients by powers of the maximal irrelevant ideal correspond precisely to the pure diagrams given
by the degree sequences $(0, d, d+1, \ldots, d+k)$.

\begin{remark}\label{rem:bsd-mm}
    Let $A = K[x_1, \ldots, x_d] / \mm^{t+1}$. The Betti table associated to $h_A$ in Theorem~\ref{thm:RI-decomp} is the pure diagram
    $\frac{(d+t)!}{(d-1)!} \cdot \pi(0,d,d+1,\ldots,d+t)$.  This is the Betti table of $K[x_1,\ldots,x_{t+1}]/\mathfrak{m}^{d}$.
\end{remark}

Moreover, we easily see the decomposition in terms of normalized pure diagrams.
This result was independently described by S\"oderberg~\cite[Proposition~3.4]{So} using
a different lexicon.

\begin{corollary}\label{cor:RI-decomp}
    Let $I$ be an ideal of $R = K[x_1, \ldots, x_n]$ with a $d$-linear resolution.
    If
    \[
        \beta(I) = \sum_{j=0}^{t} \alpha_j j! \cdot \pi(d,d+1,\ldots,d+j),
    \]
    where $\alpha_0, \ldots, \alpha_t$ are rational numbers and $t \leq n$,
    then
    \[
        \beta(R/I) = \sum_{j=1}^{t} 
            \left(
                \frac{\alpha_j}{\binom{d+j-1}{j}} - \frac{\alpha_{j+1}}{\binom{d+j}{j+1}}
            \right)
        \cdot \overline{\pi}(0,d,d+1,\ldots,d+j).
    \]
\end{corollary}
\begin{proof}
    Since $\overline{\pi}(0,d,d+1,\ldots,d+j) = \frac{(d+j)!}{(d-1)!} \cdot \pi(0,d,d+1,\ldots,d+j)$, this
    follows immediately from Theorem~\ref{thm:RI-decomp}.
\end{proof}

\begin{example}\label{exa:bsd'-decomp}
    Continuing Example~\ref{exa:bsd-decomp}, we have that $\beta(R/I)$ is associated to the 
    $O$-sequence $\alpha = (1,3,2)$ with $d = 3$.  We thus see that the decomposition of
    $\beta(R/I)$ using normalized pure diagrams has coefficients
    \begin{equation*}
        \begin{split}
            j=0\colon & \frac{1}{\binom{3+0-1}{0}} - \frac{3}{\binom{3+1-1}{1}} = 0 \\
            j=1\colon & \frac{3}{\binom{3+1-1}{1}} - \frac{2}{\binom{3+2-1}{2}} = \frac{1}{3} \\
            j=2\colon & \frac{2}{\binom{3+2-1}{2}} - \frac{0}{\binom{3+3-1}{3}} = \frac{2}{3} \\
        \end{split}
    \end{equation*}
    That is, $\beta(R/I) = \frac{1}{3} \cdot \overline{\pi}(0,3,4,5) + \frac{2}{3} \cdot \overline{\pi}(0,3,4)$.
\end{example}

% -- Subsection
\subsection{Simplicial cones of finite \texorpdfstring{$O$}{O}-sequences}\label{sub:cone}~

The decomposition in Theorem~\ref{thm:RI-decomp} provides some deceptively simple bounds on the entries of an $O$-sequence.
These inequalities were recently and independently discovered by Boij and Smith~\cite[Theorem~2.1]{BSm} using two approaches that 
differ from the one presented here.  Moreover, Bertone, Nguyen, and Vorwerk~\cite[Corollary~4.18]{BNV} provided an analogous result
in the case of modules of an exterior algebra.

\begin{lemma}\label{lem:halfspaces}
    Let $\alpha$ be an $O$-sequence.  For all nonnegative integers $j$, $$(\alpha_1 + j) \alpha_k \geq (j + 1) \alpha_{j+1}.$$
\end{lemma}
\begin{proof}
    By Theorem~\ref{thm:ns-36}, there is an ideal $I \subset R$ such that $\alpha_0, \ldots, \alpha_t$ are the Boij-S\"oderberg
    coefficients of the decomposition of $\beta(I)$.  Hence, ignoring the zero coefficients in Theorem~\ref{thm:RI-decomp}, 
    we have the Boij-S\"oderberg decomposition of $\beta(R/I)$.  By Theorem~\ref{thm:BS-decomp} the remaining coefficients
    are positive.  That is, $(\alpha_1 + j) \alpha_j - (j+1) \alpha_{j+1} \geq 0$ for nonnegative integers $j$.
\end{proof}

Alternatively, the preceding lemma follows from Corollary~\ref{cor:nzc}.

\begin{remark}\label{rem:oseq-mm}
    Let $j$ be a nonnegative integer.  In this case, we have $(d + j) \binom{d+j-1}{j-1} = (j+1) \binom{d+j}{j}$.  Hence
    the $O$-sequence $h_A = (\alpha_0, \ldots, \alpha_t)$, where $A = K[x_1, \ldots, x_d] / \mm^{t+1}$, is on the hyperplane
    $(\alpha_1 + j) \alpha_j = (j + 1) \alpha_{j+1}$ if $t > j$.  That is to say, the bounds in the preceding lemma are sharp.
\end{remark}

Moreover, combining Theorem~\ref{thm:RI-decomp} and Remark~\ref{rem:bsd-mm}, we can easily see the unique decomposition of a
$O$-sequence into a countable nonnegative linear combination of $O$-sequences of powers of the maximal irrelevant ideal.  

\begin{lemma}\label{lem:oseq-decomp}
    If $\alpha = (\alpha_0, \alpha_1, \ldots)$ is an $O$-sequence of a quotient of $R = K[x_1, \ldots, x_d]$, then
    \[
        \alpha = \sum_{j=0}^{\infty} 
            \left(
                \frac{\alpha_j}{\binom{d+j-1}{j}} - \frac{\alpha_{j+1}}{\binom{d+j}{j+1}}
            \right)
        \cdot h_{R/\mm^{j+1}}
    \]
    is the unique countable nonnegative linear combination of the $O$-sequences $h_{R/\mm^{j+1}}$, where $0 \leq j \leq t$.
    Moreover, the sum of the coefficients is $1$.
\end{lemma}

\begin{remark}
    The preceding lemma can also be derived directly from the Hilbert functions, as $h_{R}$ is greater than $\alpha$ in each entry.
    Moreover, Bertone, Nguyen, and Vorwerk \cite[Proposition~4.14]{BNV} gave an analogous result in the case of Stanley-Reisner rings.
\end{remark}

\begin{example}\label{exa:oseq-decomp}
    Continuing Example~\ref{exa:bsd'-decomp}, we have that $\beta(R/I)$ is associated to the 
    $O$-sequence $\alpha = (1,3,2)$ with $d = 3$.  Moreover, $\beta(R/I) = \frac{1}{3} \cdot \overline{\pi}(0,3,4,5) + \frac{2}{3} \cdot \overline{\pi}(0,3,4)$.
    Thus we see that 
    \[
        (1,3,2) = \frac{1}{3} \cdot h_{R/\mm^{2+1}} + \frac{2}{3} \cdot h_{R/\mm^{1+1}} 
                = \frac{1}{3} \cdot (1,3,6) + \frac{2}{3} \cdot (1,3).
    \]
\end{example}

Combining the preceding results, we can describe the cone of $O$-sequences.
This cone was recently and independently described by Boij and Smith~\cite[Theorem~1.1]{BSm} using a different
approach than then one presented here; moreover, they explored several related cones in great detail.

\begin{theorem}\label{thm:cone}
    If $C$ is the convex hull of $O$-sequences of quotients of $R = K[x_1, \ldots, x_d]$, then
    \begin{enumerate}
        \item $C$ is the intersection of the half-spaces $(d + j) \alpha_j \geq (j + 1) \alpha_{j+1}$ for integers $j \geq 0$;
        \item the extremal rays of $C$ are given by the $O$-sequences $h_{R/\mm^{j+1}}$ for $j \geq 0$; and
        \item $C$ is a closed simplicial cone.
    \end{enumerate}
\end{theorem}
\begin{proof}
    The three claims follow from Lemma~\ref{lem:halfspaces}, Remark~\ref{rem:oseq-mm}, and Lemma~\ref{lem:oseq-decomp}, respectively.
\end{proof}

S\"oderberg~\cite[Theorem~4.7]{So} described statement (i) in a different and distinct manner.  Moreover, Bertone, Nguyen, and 
Vorwerk~\cite[Corollary~4.15]{BNV} provided an analogous statement to (ii) for the Hilbert functions of Stanley-Reisner rings.

We close by noting that the space of finite $O$-sequence with degree $1$ entry bounded by $d$ and socle degree bounded by $t$ is isomorphic to the space of 
$d$-linear Betti tables with projective dimension bounded by $t$.  Moreover, this space is $t+1$ dimensional.  This follows 
as the decompositions in Corollary~\ref{cor:RI-decomp} and Lemma~\ref{lem:oseq-decomp} are the same, up to replacing the Betti table of a
power of a maximal irrelevant ideal with the $O$-sequence of a power of a maximal irrelevant ideal (see Remark~\ref{rem:bsd-mm}).

% -----------------------------------------------------------------------------
% -- Acknowledgement
\begin{acknowledgement}
    We thank Mats Boij and Greg Smith for their helpful comments.
\end{acknowledgement}

% -----------------------------------------------------------------------------
% -- The Bibliography

% -----------------------------------------------------------------------------
% -- Happy, Happy, Joy, Joy!
\end{document}